\documentclass[10pt,a4paper,oneside]{amsproc}
\usepackage[utf8]{inputenc}
\usepackage{latexsym, amsthm}
\usepackage{amsfonts, amsmath, amssymb,comment, mathtools}
\usepackage{euscript,mathrsfs}
\usepackage{url}
\usepackage{array}
\usepackage[a4paper]{geometry}
\usepackage{graphics}
\usepackage[usenames]{color}

\DeclareFontFamily{U}{mathx}{}
\DeclareFontShape{U}{mathx}{m}{n}{<-> mathx10}{}
\DeclareSymbolFont{mathx}{U}{mathx}{m}{n}
\DeclareMathAccent{\widehat}{0}{mathx}{"70}
\DeclareMathAccent{\widecheck}{0}{mathx}{"71}

\newtheorem{theorem}{Theorem}[section]

\newtheorem{corollary}[theorem]{Corollary}
\newtheorem{question}[theorem]{Question}

\newtheorem{remark}[theorem]{Remark}

\newtheorem{lemma}[theorem]{Lemma}

\newtheorem{proposition}[theorem]{Proposition}

\newcommand{\B}{\mathcal{B}}

\newcommand{\ZZ}{\mathbb{Z}}

\newcommand{\Fq}{\mathbb{F}_q}

\newcommand{\M}{\mathcal{M}}

\def\a{{\alpha}}
\def\Fq{{\mathbb F}_q}

\def\FF{{\mathbb F}}
\def\PP{{\mathbb P}} 
\newcommand{\V}{{\mathcal{V}}}

\begin{document}
\title{The second minimum weight of Grassmann codes}  

\author{Mrinmoy Datta}

\address{Department of Mathematics, \newline \indent
Indian Institute of Technology Hyderabad, Kandi, Sangareddy, Telangana, India}
%\curraddr{}
\email{mrinmoy.datta@math.iith.ac.in}
\thanks{The first named author is partially supported by a research grant 02011/34/2025/R\&D-II/DAE/9465 from the National Board for Higher Mathematics, Department of Atomic Energy, India.} %and the grant SRG/2021/001177 from Science and Engineering Research Board, Govt. of India}

\author{Tiasa Dutta}

\address{Department of Mathematics, \newline \indent
Indian Institute of Technology Hyderabad, Kandi, Sangareddy, Telangana, India}
%\curraddr{}
\email{duttatiasa98@gmail.com}
\thanks{The second named author is partially supported by a DST-INSPIRE Ph.D. fellowship from the
Department of Science and Technology, Govt. of India}
\keywords{Finite fields, Hyperplanes, Rational points, Grassmann codes}
\subjclass[2020]{Primary 14G50, 14M15, 94B27}
%%\subjclass[2010]{11T35, 11T06 20G40, 15B05}
%%    The 2010 edition of the Mathematics Subject Classification is
%%    now available.  If you are citing a classification from the
%%    new scheme, use the following input coding instead.
%%\subjclass[2010]{Primary }

\date{}
\begin{abstract}   
We give an independent combinatorial proof of Nogin's Theorem concerning the minimum distance of the Grassmann codes using a special decomposition of the Grassmannians. We use the same idea to also compute the second minimum weight of the Grassmann codes. 
\end{abstract}
\maketitle

\section{Introduction}

%Grassmannians and their Schubert subvarieties have been widely studied in several streams of mathematics, and in particular in algebraic geometry and combinatorics. One of the

Fix a prime power $q$ and the finite field $\Fq$ with $q$ elements. 
An $[n, k]_q$ linear code (or simply an $[n, k]$ code) is a 
$k$-dimensional subspace of $\Fq^n$. On $\Fq^n$, viewed as an 
$n$-dimensional vector space with the natural basis, one can 
naturally define the Hamming metric in the following way: Given 
$x = (x_1, \dots, x_n), y = (y_1, \dots, y_n) \in \Fq^n$, we 
define the Hamming distance between $x$ and $y$, denoted by 
$d_H(x, y)$, as $d_H(x, y) = |\{i : x_i \neq y_i\}|$. If $C$ 
is an $[n, k]$ code, then the minimum distance or minimum weight 
of $C$ is defined as
$$d(C) := \min\{d_H(x, y) : x, y \in C, x \neq y\} = 
\min\{d_H(x, 0) : x \in C, x \neq 0\}.$$
An $[n, k]$ code with minimum distance $d$ is called an $[n, k, d]$ 
code.

Let $V_m$ be a vector space over $\Fq$ of dimension $m$. The 
Grassmannian $G(\ell, V_m)$ is the set of all $\ell$-dimensional 
subspaces of $V_m$. As explained in Section~\ref{sec:gra}, the 
Grassmannian $G(\ell, V_m)$ can be embedded in the projective space 
$\PP(\bigwedge^\ell V_m)$ via the Pl\"ucker embedding 
$\pi_{\ell,m}: G(\ell, V_m) \to \PP(\bigwedge^\ell V_m)$. 
Fix an ordering of the elements of $G(\ell, V_m) = 
\{L_1, \ldots, L_n\}$, where $n = {{m} \brack{\ell}}_q$. 
For any linear functional $f \in \left(\bigwedge^\ell V_m\right)^*$, 
we define the evaluation map
$$\mathrm{ev}: \left(\bigwedge^\ell V_m\right)^* \to \Fq^n 
\quad \text{given by} \quad 
f \mapsto \left(f(\pi_{\ell,m}(L_1)), \ldots, 
f(\pi_{\ell,m}(L_n))\right).$$
Since $G(\ell, V_m)$ is a non-degenerate subset of 
$\PP(\bigwedge^\ell V_m)$ (as explained in 
Section \ref{sec:hyp}), the map $\mathrm{ev}$ is injective.
Consequently, the image of $\mathrm{ev}$ is a linear code of length 
$n = {m \brack \ell}_q$ and dimension $k = \binom{m}{\ell}$, 
called the \emph{Grassmann code} and denoted by $C(\ell, m)$. 
%These codes were first studied in \cite{R, R1, R2}.

Moreover, the Hamming weight of the codeword $\mathrm{ev}(f)$ equals 
$n - |\Pi_f \cap G(\ell, V_m)|$, where $\Pi_f$ is the hyperplane 
defined by $f$ in $\PP(\bigwedge^\ell V_m)$. Consequently, the 
number of weight-$w$ codewords of $C(\ell, m)$ is in bijection 
with the set of hyperplanes $\Pi$ of $\PP(\bigwedge^\ell V_m)$ 
satisfying $|\Pi \cap G(\ell, V_m)| = n - w$. In particular, 
determining the minimum distance of $C(\ell, m)$ is equivalent 
to determining
$$\max\left\{|\Pi \cap G(\ell, V_m)| : \Pi \ \text{is a hyperplane of} 
\ \PP(\bigwedge^\ell V_m)\right\}.$$

 These codes were first studied in \cite{R, R1, R2}. In \cite{N}, Nogin showed that the minimum distance of $C(\ell, m)$ is given by $q^{\ell (m - \ell)}$. Equivalently, if $\Pi$ is a hyperplane in $\PP\left(\bigwedge^\ell V_m \right)$, then 
\begin{equation}\label{Nog}
|\Pi \cap G(\ell, V_m)| \le {m \brack \ell}_q - q^{\ell (m - \ell)}.
\end{equation}
Nogin also proved that the upper bound in \eqref{Nog} is attained by a hyperplane $\Pi$ if and only if $\Pi$ is decomposable (cf. Section \ref{sec:hyp}). This gives a complete classification of the minimum weight codewords of Grassmann codes. 
In this article, we are interested in determination of the second minimum weight of the Grassmann code $C(\ell, m)$. Geometrically, we ask the following question.
\begin{question}\label{q1}
   Determine $$\max \left\{|\Pi \cap G(\ell, V_m)| : \Pi \ \text{is a hyperplane in} \ \PP\left(\bigwedge^\ell V_m \right), |\Pi \cap G(\ell, V_m)| < {m \brack \ell}_q - q^{\ell (m - \ell)} \right\}.$$
\end{question}
 
To the best of our knowledge,
a general formula for the second minimum weight of $C(\ell, m)$ is not known. We remark that the complete weight distribution for $C(\ell, m)$, when $\ell = 2$ was obtained by Nogin in \cite{N}. Later Nogin also determined the complete weight distribution for $C(3, 6)$ in \cite{N1}. Morever, Kaipa and Pillai worked out the complete weight distribution for $C(3, 7)$ in \cite{KP}.  In general, the Grassmannian codes have been widely studied. For example, a few of the generalized Hamming weights are determined in \cite{N, GL, GPP, HJR}, while the automorphism groups of Grassmann codes were determined in \cite{GK}.

A generalization of Grassmann codes, namely, the Schubert codes 
were introduced by Ghorpade and Lachaud in \cite{GL}. To this end, 
fix an ordered basis $\B = \{v_1, \dots, v_m\}$ of $V_m$ and an 
ordering of the elements of $I(\ell, m)$, as explained in 
Subsection~\ref{sec:gra} and \ref{sec:sch}. For any $\a \in I(\ell, m)$, the 
$\a$-th Schubert subvariety $\Omega_\a$ is embedded as a 
non-degenerate subset of $\PP^{k_\a - 1}$, where 
$k_\a = |\nabla(\a)|$. Fix an ordering of the points of $\Omega_\a 
= \{P_1, \ldots, P_{n_\a}\}$, where 
$$n_\a = \sum_{\beta \in \nabla(\a)} q^{\delta(\beta)}.$$
For any linear functional $f$ on the ambient projective space 
$\PP^{k_\a - 1}$, we define the evaluation map
$$\mathrm{ev}_\a : (\PP^{k_\a - 1})^* \to \Fq^{n_\a} \quad 
\text{given by} \quad f \mapsto (f(P_1), \ldots, f(P_{n_\a})).$$
The image of $\mathrm{ev}_\a$ is called the \emph{Schubert code} 
and is denoted by $C_\a(\ell, m)$. For the definitions of 
$I(\ell, m)$, $\delta(\beta)$, and $\nabla(\a)$, we refer the 
reader to Subsection~\ref{sec:sch}. Several interesting formulas 
for $n_\a$ and $k_\a$ are known. We refer the reader to \cite{GT} 
and the references therein for more information on the same. Note 
that when $\a = (m-\ell+1, \ldots, m)$, the code $C_\a(\ell, m)$ 
is the same as the Grassmann code $C(\ell, m)$. 

The minimum distance of the Schubert code $C_\a(\ell, m)$ is given 
by $q^{\delta(\a)}$, where 
$\delta(\a) = \a_1 + \cdots + \a_\ell - \frac{\ell(\ell+1)}{2}$. 
Geometrically speaking, this is equivalent to the fact that if 
$\Pi$ is a hyperplane in $\PP^{k_\a - 1}$, then 
\begin{equation}\label{minsch}
|\Pi \cap \Omega_\a| \le n_\a - q^{\delta(\a)}.
\end{equation}
A minimum distance formula for the code $C_\a(\ell, m)$ was 
conjectured in \cite{GL}. This conjecture was proved when $\ell = 2$ 
by Chen in \cite[Theorem 2.1]{C} and independently by Guerra and 
Vincenti in \cite[Theorem 1.2]{GV}. Finally, Xiang settled the 
conjecture in the affirmative in \cite[Theorem 2]{X}. Recently, 
a new proof of this conjecture was published by Ghorpade and Singh 
in \cite[Theorem 3.6]{GS}. However, unlike in the case of Grassmann 
codes, a complete description of the minimum weight codewords of 
Schubert codes is still unknown in general. Ghorpade and Singh 
\cite[Conjecture 5.6]{GS} have conjectured that the minimum weight 
codewords of the Schubert codes are given by the so-called 
\emph{Schubert decomposable} codewords. A description of the 
Schubert decomposable codewords is beyond the scope of this paper 
and we refer the reader to \cite{GS} for the same. Recently, 
considerable progress in the determination of the minimum weight 
codewords of Schubert codes has been made in \cite{DDJ}, where 
the conjecture of Ghorpade and Singh is established for all 
Schubert varieties for all but finitely many values of $q$.

This paper is organized as follows. In Section~\ref{sec:prel}, 
we give a rather leisurely introduction to the basics of 
Grassmannians and their Schubert subvarieties. In 
Section~\ref{dec}, we give a combinatorial decomposition of the 
Grassmannians which, to the best of our knowledge, is new. 
Finally, we answer Question~\ref{q1} in Section~\ref{sec:main} 
followed by providing an independent proof of Nogin's Theorem on 
the minimum distance of Grassmann codes.

\section{Preliminaries}\label{sec:prel}
In this section, we recall the definition of Grassmannian, their Schubert subvarieties, along with the the codes obtained from them. We also collect several of the properties of Grassmannian and Schubert varieties in this section for the ease of reference. An expert could simply skip this section and proceed straight to the subsequent sections. We believe that this rather elaborate section will be useful for a reader having somewhat less experience working with Grassmannian. As such, none of the results in this section are new. They can be found in classical texts such as \cite{A, HP, M}, and in modern texts such as \cite{LB}. Moreover, the results on coding theory that are mentioned here could be found in \cite{C, GL, GPP, GS, HJR, N, X}. We will try to provide accurate references to facts as we go along with the discussion.  

\subsection{Grassmannian, Pl\"ucker projective space, and Pl\"ucker embedding}\label{sec:gra}
Let $\ell, m$ be integers satisfying $1 \le \ell \le m$ and $V_m$ be an $m$-dimensional vector space over a field $\mathbb{F}$ with a fixed ordered basis $\mathcal{B} = \{v_1, \dots, v_m\}$ of $V_m$. 
%Define $k = {m \choose \ell}$. Let $V_m$ be a vector space of dimension $m$ over $\Fq$. 
We denote by $G (\ell, V_m)$ the \emph{Grassmannian of all $\ell$-dimensional subspaces of $V_m$}, that is,
$$G (\ell, V_m) :=\{ L \subset V_m : L \ \text{is a subspace of} \ V_m, \ \dim L = \ell\}.$$
Having defined $G(\ell, V_m)$ set-theoretically, we now proceed to 
impose a geometric structure on $G(\ell, V_m)$ via the well-known \textit{Pl\"{u}cker embedding} that is defined as follows:
Define 
\begin{equation}\label{pl1}
    \pi_{\ell, m} : G(\ell, V_m) \to \PP \left(\bigwedge^\ell V_m\right) \ \ \text{given by} \ \ L = {\rm Span} \{\omega_1, \dots, \omega_\ell\} \mapsto [\omega_1 \wedge \cdots \wedge \omega_\ell].
\end{equation}
 The projective space $\PP \left(\bigwedge^\ell V_m\right)$ is called the \textit{Pl\"ucker projective space}.
 The map $\pi_{\ell,m}$ is well-defined since replacing some $\omega_j$ by $a\omega_j + b\omega_i$ for some $\omega_i$ and for any $a, b \in \mathbb{F}$ with $a \neq 0$ leaves the class 
$[\omega_1 \wedge \cdots \wedge \omega_\ell]$ unchanged, and any 
two bases of $L$ can be obtained from one another by a sequence 
of such elementary operations. Moreover, the map $\pi_{\ell,m}$ 
is injective \cite[Theorem 5.2.1]{LB}.
 
 Let us denote by 
$$I(\ell, m) = \{\a = (\a_1, \dots, \a_\ell) \in \ZZ^\ell : 1 \le \a_1 < \cdots < \a_\ell \le m\}.$$ That is, the set $I(\ell, m)$ consists of all the subsets of $\{1, \dots, m\}$ of size $\ell$ written in increasing order. The elements are expressed as $(\a_1, \dots, \a_\ell)$ as opposed to $\{\a_1, \dots, \a_{\ell}\}$ to emphasize that the coordinates are written in increasing order.
Fix an ordering of the elements of $I(\ell, m)$. This in turn induces an ordering of $\{X_\a : \a \in I(\ell, m)\}$, the homogeneous coordinates of the projective space $\PP^{k-1}$.
Note that $\bigwedge^{\ell} V_m$ is a $k = {m \choose \ell}$-dimensional vector space over $\mathbb{F}$ with an ordered basis $\{v_\a = v_{\a_1} \wedge \cdots \wedge  v_{\a_\ell} : \a \in I(\ell, m)\}.$ Let $L \in G(\ell, V_m)$ and $\{\omega_1, \dots, \omega_\ell \}$ be a basis of $L$ over $\FF$. Since $\{v_1, \ldots, v_m\}$ is a basis of $V_m$, we can express each $\omega_i$ as a linear combination $$\omega_i = a_{i1}v_1 + \cdots + a_{im}v_m,$$ and we call the coefficients $a_{ij} \in \mathbb{F}$ the coordinates of $\omega_i$ with respect to $\mathcal{B}$.
In particular, there is an $\ell \times m$ matrix $A_L = (a_{ij})$  such that the row space of $A_L$ equals $L$. It can be shown that 
\begin{equation}\label{coord}
    \omega_1 \wedge \cdots \wedge \omega_\ell = \sum_{\a \in I(\ell, m)} p_{\a} (A_L) v_\a,
\end{equation}
where $p_{\a} (A_L)$ is the $\a$-th minor of $A_L$. The element $[p_\a (A_L)]_{\a \in I(\ell, m)} \in \PP (\bigwedge^{\ell} V_{m})$ is called the Pl\"ucker coordinates of $L$ with respect to the basis $\B$. By abuse of notation, we keep denoting the image $\pi_{\ell, m} (G(\ell, V_m))$ by $G(\ell, V_m)$, which does not cause any ambiguity.
%It is well-known (cf. \cite[Section 5.2.1]{LB}) that $\pi_{\ell, m}$ is independent of the choice of basis of $L$; that is, the map $\pi_{\ell, m}$ is well-defined. Moreover, the map  $\pi_{\ell, m}$ is injective \cite[Theorem 5.2.1]{LB}. 

It is worth mentioning that all the above discussions are subject to a fixed basis of $V_m$. \emph{What happens when we change the basis of $V_m$?} It is well-known that a change of basis of $V_m$, that is, the application of an invertible linear map 
$\phi: V_m \to V_m$, induces an automorphism $\bigwedge^\ell \phi: \bigwedge^\ell V_m \to \bigwedge^\ell V_m$ 
given by $v_{\alpha_1} \wedge \cdots \wedge v_{\alpha_\ell} \mapsto \phi(v_{\alpha_1}) \wedge \cdots \wedge \phi(v_{\alpha_\ell})$, 
known as the $\ell$-th compound of $\phi$ (see \cite[Page 94]{M}). Consequently, any change in the basis of $V_m$ induces a 
projective linear isomorphism of $\PP(\bigwedge^\ell V_m)$ taking $G(\ell, V_m)$ to itself (see \cite[Chapter XIV, 
Section 1]{HP}).
Thus, a change in the basis of $V_m$ does not essentially affect the image of the Pl\"ucker map.  One advantage of using the description in \eqref{coord} for elements of $G(\ell, V_m)$ is that it allows us to introduce the projective coordinates $(X_\a)_{\a \in I(\ell, m)}$ for the projective space $\PP \left(\bigwedge^\ell V_m\right)$ when an ordered basis $\{v_1, \dots, v_m\}$ of $V_m$ and an ordering of the elements in $I(\ell, m)$ is fixed. When $\FF$ is an algebraically closed field, it can be shown that $G(\ell, V_m)$ is given by the solutions to a system of some homogeneous quadratic equations in $X_\a$-s, known as the Plücker relations. We refer to \cite[Theorem 5.2.3]{LB} for a complete proof of this fact. Consequently, the subset $G(\ell, V_m)$ can be regarded as a projective algebraic variety when working over an algebraically closed field.

At any rate, let us keep working with a fixed ordered basis $\B = \{v_1, \dots, v_m\}$ of $V_m$, as mentioned above. Since a change of basis of $L \in G(\ell, V_m)$ does not affect $\pi_{\ell, m}$, we may choose a basis of $L$ such that $A_L$ is in right-row-reduced echelon form (that is, the echelon form where the pivot of each row is its rightmost nonzero entry, the pivots move from left to right as we go down the rows, and all entries above and below each pivot are zero), which is uniquely determined by $L$. That is, $L$ can be uniquely represented by an $\ell \times m$ matrix  $M_{\B}(L)$ satisfying
\begin{enumerate}
    \item[(a)] the rows of $M_{\B}(L)$ are elements of $L$ written with respect to the basis $\mathcal{B}$ of $V_m$.
    \item[(b)]  the row-space of $M_{\B}(L)$ is $L$, 
    \item[(c)] for each $i = 1, \dots, \ell$ the last non-zero entry of the $i$-th row, called the \emph{pivot of $i$-th row}, is equal to $1$,
    \item[(d)] the last non-zero entry of $(i+1)$-st row appears to the right of the last nonzero entry of the $i$-th row,
    \item[(e)] all the entries above and below of a pivot are $0$. 
\end{enumerate}
 In particular,  this sets up a one-to-one correspondence,
\begin{equation}\label{em}
    G(\ell, V_m) \longleftrightarrow \M (\ell, m) \ \ \text{given by} \ \ L \longleftrightarrow M_{\B}(L),
\end{equation}
where $\M (\ell, m)$ is the set of all $\ell \times m$ matrices in right-row-reduced echelon form with entries in $\FF$. 
It follows from the discussion above that, with respect to a fixed basis $\B$ of $V_m$, the map $\pi_{\ell, m}$ in \eqref{pl1} can be expressed in terms of coordinates  as 
\begin{equation}\label{pl2}
    \pi_{\ell, m}: G (\ell, V_m) \to \PP\left(\bigwedge^\ell V_m\right) \ \ \text{as} \ \ \pi_{\ell, m} (L) = (p_\a (M_{\B}(L)))_{\a \in I(\ell, m)}.
\end{equation}
This description will be used later in this article. 

\subsection{Schubert subvarieties of $G(\ell, V_m)$}\label{sec:sch} The Grassmannian $G(\ell, V_m)$ contains a special class of 
subvarieties, known as the \emph{Schubert varieties}. Following 
the notations introduced in the previous subsection, we recall 
the definitions and a few interesting properties of Schubert 
varieties that will be used later in this paper. We define the 
partial order, also known as the Bruhat order in the present 
context, on the elements of $I(\ell, m)$ as follows: for 
$\alpha = (\alpha_1, \ldots, \alpha_\ell), \beta = (\beta_1, 
\ldots, \beta_\ell) \in I(\ell, m)$, we say 
$$\alpha \leq \beta \iff \alpha_i \leq \beta_i \ \text{ for all } 
\ i \in \{1, \ldots, \ell\},$$
that is, the order is defined by entry-wise comparison, also 
known as the product order on $\mathbb{Z}^\ell$ restricted to 
$I(\ell, m)$. For $\alpha \in I(\ell, m)$, define 
$$\nabla(\alpha) := \{\beta \in I(\ell,m) : \beta \leq \alpha\},$$
that is, $\nabla(\alpha)$ is the \emph{downset} generated by 
$\alpha$ in $I(\ell, m)$ with respect to the Bruhat order. We 
set $\Delta(\alpha) = I(\ell, m) \setminus \nabla(\alpha)$ and 
$k_\alpha = |\nabla(\alpha)|$. For $\a \in I(\ell, m)$, the 
\emph{$\a$-th Schubert cell}, with respect to the basis 
$\B =\{v_1, \dots, v_m\}$, denoted by $C_{\a}$\footnote{Ideally, 
we should use a more complicated notation such as 
$C_\a (\ell, m, \B)$ to emphasize the dependence on the 
parameters $\ell$, $m$, and the basis $\B$. However, in the 
subsequent parts of this article, there will be no ambiguity 
if we restrict to this simpler notation}, is defined as
$$C_{\a} := \{L \in G(\ell, V_m) : \ \text{the pivots of} \ 
M_{\B}(L) \ \text{are on the columns} \ \a_1, \dots, \a_\ell \}.$$
It is easy to see that $C_\a$ could be identified with an affine 
space over $\FF$ of dimension $\delta (\a)$, where 
$\delta (\a) = (\a_1 + \dots + \a_\ell) - \frac{\ell (\ell+1)}{2}$.

The \emph{$\a$-th Schubert variety in} $G(\ell, V_m)$ with respect to the basis $\B$, denoted by $\Omega_\a$, is defined as 
\begin{equation}\label{Sch}
\Omega_\a = \bigcup_{\beta \in \nabla(\a)} C_{\beta}.
\end{equation}
In particular, if $\a = (m - \ell + 1, \dots, m)$, then $\Omega_\a = G(\ell, V_m)$. 
%We remark that the definitions of Schubert cells and Schubert varieties given above are consistent with those in \cite[Section 2]{HJR}. However, there are equivalent ways of defining the Schubert varieties using the notion of flags, which can be found, for example, in \cite[Section 5.2.3]{LB}. 
For $i= 1, \dots, m$, let $V_i$ be the subspace of $V_m$ spanned by $\{v_1, \dots, v_i\}$. We see that for every $\a = (\a_1, \dots, \a_\ell) \in I(\ell, m)$, $$L \in C_{\a} \iff \dim L \cap V_{\a_j} = j  \ \text{and} \dim L \cap V_i < j \ \text{for} \ i < \a_j \ \  \text{for every} \ j = 1, \dots, \ell,$$ and  $$L \in \Omega_{\a} \iff \dim L \cap V_{\a_j} \ge j \ \  \text{for every} \ j = 1, \dots, \ell.$$

Geometrically speaking, the Pl\"ucker embedding $\pi_{\ell, m}$ restricted to $\Omega_\a$ allows us to view the Schubert variety as a subset of the projective space $\PP\left(\bigwedge^{\ell} V_m\right)$. Recall that, with respect to the basis $\B$ and a fixed ordering of the elements of $I(\ell, m)$, the projective space $\PP^{k-1} = \PP\left(\bigwedge^{\ell} V_m\right)$ is endowed with the homogeneous coordinates $(X_{\a})_{\a \in I(\ell, m)}$. Define the projective linear subspace $\PP^{k_\a - 1}$ of $\PP\left(\bigwedge^{\ell} V_m\right)$ by 
$$\PP^{k_\a - 1} := \left\{ P \in \PP\left(\bigwedge^{\ell} V_m\right) : X_\beta (P) = 0 \ \text{for all} \ \beta \in \Delta (\a)\right\},$$
where $\{X_{\beta}: \beta \in I(\ell, m)\}$ are the homogeneous coordinates of $\PP^{k-1} = \PP \big(\bigwedge^{\ell} V_m \big)$, as defined in the last subsection.  
It is  well-known that 
\begin{equation}\label{commd1}
    \pi_{\ell, m} (\Omega_\a)  = G(\ell, V_m) \cap \PP^{k_a - 1},
\end{equation}
%= \{L \in G(\ell, V_m)  : p_{\beta} (L) = 0, \ \text{for all} \ \beta \in \Delta(\a)\}
where $\PP^{k_\a - 1}$ is as defined above.
That is, the $\a$-th Schubert variety $\Omega_\a$ can be geometrically understood as a linear section of the Grassmann variety $G(\ell, V_m)$ given by some coordinate hyperplanes of the Pl\"ucker space. As with the Grassmannian $G(\ell, V_m)$, we will use the notation $\Omega_\a$ to denote its image in the Pl\"ucker projective space $\PP^{k_\a-1}$.  %We denote by $\PP^{k_\a - 1}$ the linear subspace of $\PP^{k-1}$ given by $$\PP^{k_\a - 1} := \{P \in \PP^{k-1} : X_{\beta} (P) = 0 \ \ \text{for all} \ \beta \in \Delta(\a)\}.$$ The discussion above can now be summarized using the following commutative diagram. 

\subsection{Hyperplanes in Pl\"ucker projective space}\label{sec:hyp}  Generally speaking, if $V$ is a vector space over a field $\FF$, then the hyperplanes of the projective space $\PP (V)$ are given by the nonzero elements of $V^*$, the dual space of $V$. As can be found in any standard linear algebra textbook (for example see \cite[Section 3.5]{HK}), if $\{v_1, \dots, v_m\}$ is an ordered basis of $V$, then $\{v_1^*, \dots, v_m^*\}$ is an ordered \textit{dual} basis of $V^*$, that is
\begin{align*}
    v_i^* (v_j) =
    \begin{cases}
        0 \ \ &\text{if} \ i \neq j \\
        1 \ \ &\text{if} \ i = j.
    \end{cases}
\end{align*}
In particular, any element of $f \in V^*$ can be expressed uniquely as 
\begin{equation}\label{dualformula}
f = \sum_{i=1}^m f(v_i) v_i^*.\end{equation}
%A proof of the above statement can be found in \cite[Section 3.5, Theorem 15]{HK}.

Let us now return to the study of Grassmannians. As in the previous subsections, we fix  vector space $V_m$ of dimension $m$ over a field $\FF$, an ordered basis $\B = \{v_1, \dots, v_m\}$ of $V_m$ and a fixed ordering of the elements of $I(\ell, m)$. We now have an ordered basis $\B_\ell = \{v_\a : \a \in I(\ell, m)\}$  of $\bigwedge^\ell V_m$, where for any $\a = (\a_1, \dots, \a_\ell) \in I(\ell, m)$, the element $v_\a = v_{\a_1} \wedge \cdots \wedge v_{\a_\ell}$. Furthermore, for any $\a = (\a_1, \dots, \a_\ell) \in I(\ell, m)$, we define $\a^{\mathsf{C}}$ to be the unique element of $I(m - \ell, m) = \{(\a_1, \dots, \a_{m-\ell}) \in \ZZ^{m-\ell} : 1 \le \a_1 < \cdots < \a_{m - \ell} \le m\}$ such that $\a \cup \a^{\mathsf{C}} = \{1, \dots, m\}$. Moreover, for each such $\a \in I(\ell, m)$, we define $\epsilon (\a)$ to be the sign of the permutation that takes the element $(\a^{\mathsf{C}}, \a)$, that is obtained by juxtaposing $\a^{\mathsf{C}}$ and $\a$, to $(1, \dots, m)$. 
It is known that the bilinear map 
\begin{equation}\label{wedge}
    \bigwedge : \bigwedge^{m-\ell} V_m \times \bigwedge^{\ell} V_m \to \FF
\end{equation}
is non-degenerate and the dual basis corresponding to $\B_{\ell}$ is given by $\B_{m - \ell} = \{{\epsilon({\a})} v_{\a^{\mathsf{C}}} : \a \in I(\ell, m)\}$. This readily gives rise to the isomorphism
\begin{equation}\label{dual}
  \left(\bigwedge^\ell V_m\right)^* \cong \bigwedge^{m-\ell} V_m.
\end{equation}

If $F \in \bigwedge^{m- \ell} V_m$ defines any hyperplane $\Pi$ in $\PP(\bigwedge^{\ell} V_m)$, then $F$ is given by a linear combination of elements of $\B_{m-\ell}$, i.e. 
\begin{equation}\label{hyperplane}
  F = \sum_{\a \in I(\ell, m)}{\epsilon(\a)} c_\a  v_{\a^{\mathsf{C}}} \in \bigwedge^{m-\ell} V_m,  
\end{equation}
where, the coefficients $c_{\a} \in \FF$ are uniquely determined by \eqref{dualformula}.  On the other hand, any hyperplane in $\PP(\bigwedge^\ell V_m)$ is given by an equation of the form 
$$\sum_{\a \in I(\ell, m)} c_\a X_\a =0,$$
where $X_\a$ are the homogeneous coordinates of the Pl\"ucker space $\PP (\bigwedge^{\ell} V_m)$. 
What does $X_\a$ correspond to as an element of $\bigwedge^{m-\ell} V_m$? Alluding to  \eqref{dualformula}, and elementary properties of exterior products, it can be readily checked that the hyperplane of $\PP(\bigwedge^\ell V_m)$ corresponding to $X_\a$ is given by the element ${\epsilon(\a)} v_{\a^{\mathsf{C}}} \in \bigwedge^{m-\ell} V_m$. Consequently, for a general hyperplane in $\PP (\bigwedge^{\ell} V_m)$, we have the following identification 
\begin{equation}\label{hypid}
    \sum_{\a \in I (\ell, m)} c_\a X_{\a} \longleftrightarrow \sum_{\a \in I(\ell, m)} {\epsilon(\a)} c_\a v_{\a^{\mathsf{C}}}.
\end{equation}
 The non-degeneracy of the map in \eqref{wedge} implies that the Grassmannian $G(\ell, V_m)$ is not contained in any hyperplane of $\PP \left(\bigwedge^{\ell} V_m\right)$. In other words, the Grassmannian $G(\ell, V_m)$ is a \emph{non-degenerate subset} of $\PP(\bigwedge^{\ell} V_m)$. It is also well-known that the $\a$-th Schubert variety $\Omega_\a$ of $G(\ell, V_m)$ is also a non-degenerate subset of $\PP^{k_\a - 1}$. For a proof of the non-degeneracy, we refer the reader to \cite[Remark 5.3.4]{LB}. The following observation will be useful later. So we record this as a remark for ease of reference.

\begin{remark}\label{rem1}\normalfont
Suppose $\Pi$ is a hyperplane in $\PP (\bigwedge^{\ell} V_m)$ given by the zeroes of a homogeneous linear polynomial $$F(X_\beta : \beta \in I(\ell, m)) = \sum_{\beta \in I(\ell, m)} c_\beta X_\beta.$$ For $\a \in I(\ell, m)$, the polynomial; $F$ restricts to a linear homogeneous polynomial $F_\a$ on $\PP^{k_\a-1}$ given by 
$$F_\a (X_\beta : \beta \in I(\ell, m)) = \sum_{\beta \in \nabla(\a)} c_\beta X_\beta.$$
The non-degeneracy of $\Omega_\a \subseteq \PP^{k_\a - 1}$ implies that $\Pi$ contains $\Omega_\a$ if and only if $c_{\beta} = 0$ for all $\beta \in \nabla(\a)$. Furthermore, if the hyperplane $\Pi$ does not contain $\Omega_\a$, then $\Pi$ restricts to the hyperplane $\Pi_\a$ in $\PP^{k_\a - 1}$ given by the equation $F_\a = 0$. 
\end{remark}

A nonzero element $z \in \bigwedge^{m- \ell} V_m$ is said to be \emph{decomposable} if there exist $w_1, \dots, w_{m-\ell} \in V_m$ such that $z = w_1 \wedge \cdots \wedge w_{m-\ell}$. A hyperplane of $\PP\left(\bigwedge^\ell V_m\right)$ is said to be \emph{decomposable} if it is given by a decomposable element of $\bigwedge^{m- \ell} V_m$. Thus, in view of \eqref{hypid}, a hyperplane $\Pi$ of $\displaystyle{\PP\left(\bigwedge^{\ell} V_m\right)}$ given by the equation $\displaystyle{\sum_{\a \in I(\ell, m)}} c_\a X_\a = 0$ is decomposable if and only if $\displaystyle{\sum_{\a \in I(\ell, m)} \epsilon (\a) c_\a v_{\a^{\mathsf{C}}}}$ is a decomposable element of $\bigwedge^{m - \ell} V_m$.  

\begin{remark}\label{rem1:dec}\normalfont
    For a nonzero element $z \in \bigwedge^{m - \ell} V_m$, define a subspace $V_m (z) = \{x \in V_m : z \wedge x = 0\}$. We have
    \begin{enumerate}
        \item[(a)] \cite[Section 4.1, Theorem 1.1]{M}  $z$ is decomposable if and only if $\dim V_m (z) = m - \ell$. 
        \item[(b)] \cite[Section 4.1, Theorem 1.3]{M} if $\ell = m-1$, then $z$ is decomposable. 
    \end{enumerate}
    Let us see a quick application of (a) and (b) in the simple case when $\ell = 2$. Let us fix a basis $\B = \{v_1, \dots, v_m\}$ of $V_m$ and consider a hyperplane $\Pi$ of $\PP\left(\bigwedge^2 V_m\right)$ given by the equation $\displaystyle{\sum_{i=1}^{m-1}} {c_i} X_{(i,m)} = 0$. In view of \eqref{hypid}, the defining element $z \in \PP \left( \bigwedge^{m- 2} V_m \right)$ of $\Pi$ could be expressed in terms of wedges as 
    $$z= \sum_{i=1}^{m-1} \epsilon ((i, m)) c_i v_1 \wedge \dots \wedge \widecheck{v_i} \wedge \dots \wedge v_{m-1},$$
    where $(i, m)^{\mathsf{C}}$ is the unique element $\{\a_1, \dots, \a_{m-2}\}$ of $I (m - 2, 2)$ such that $\{\a_1, \dots, \a_{m-2}\} \cup \{i, m\} = \{1, \dots, m\}$ and  $v_1 \wedge \dots \wedge \widecheck{v_i} \wedge \dots \wedge v_{m-1}$ is the element of $\bigwedge^{m-2} V_m$ that is obtained by first writing down the wedge $v_1 \wedge \cdots \wedge v_{m-1}$ and dropping the $i$-th term from the same. 
    It follows that $z \in \bigwedge^{m- 2} V_{m-1}$, where $V_{m-1}$ is the subspace of $V_m$ spanned by $\{v_1, \dots, v_{m-1}\}$. By (b), the element $z$ is decomposable in $\bigwedge^{m- 2} V_{m-1}$. Since $V_{m-1} \subset V_m$, we see that $z$ is also decomposable in $\bigwedge^{m- 2} V_{m}$. This can also be derived from the discussion on decomposable subspaces in \cite[Section 2]{GPP}. 
\end{remark}

\subsection{Enumeration over a finite field}\label{sec:en} Let us now restrict our attention to Grassmannians and their Schubert subvarieties over finite fields. To this end, let us fix a prime power $q$, and denote by $\Fq$ the finite field with $q$ elements. As usual, let $V_m$ be an $m$- dimensional vector space over $\Fq$ and $G(\ell, V_m)$ the Grassmannian of all $\ell$-dimensional subspaces of $V_m$. It is well-known  that the cardinality of $G(\ell, V_m)$ is given by the so-called \emph{Gaussian Binomials}, that is, 
\begin{equation}\label{gb}
   |G(\ell, V_m)| = {m \brack \ell}_q := \frac{(q^m - 1)\cdots (q^m - q^{\ell - 1})}{(q^\ell - 1)\cdots (q^\ell - q^{\ell - 1})}. 
\end{equation}
We state some of the well-known identities of the Gaussian Binomals below as a proposition for ease of reference. 

\begin{proposition}\label{identities}
    For positive integers $\ell, m$ satisfying $1 \le \ell \le m$, we have,
    \begin{enumerate}
        \item[(a)] ${m \brack \ell}_q = {m \brack m - \ell}_q$. 
        \item[(b)] ${m \brack \ell}_q = {m - 1 \brack \ell}_q + q^{m- \ell}{m - 1 \brack \ell - 1}_q$.
        \item[(c)] ${m \brack \ell}_q = \frac{q^m - 1}{q^{m - \ell} - 1} {m - 1 \brack \ell}_q$
    \end{enumerate}
\end{proposition}
We refer the reader to \cite[Section 3.2]{A} for proofs of (a) and (b) above. 
%Part (b), which is an analog of the so-called Pascal identity for binomial coefficients, also follows from \eqref{t}. 
The identity in (c) follows trivially from  \eqref{gb}. As usual, when an ordered basis $\B$ of $V_m$ over $\Fq$ and an ordering of elements of $I(\ell, m)$ is fixed,  the identification of a Schubert cell $C_\a$ with an affine space of dimension $\delta (\a)$ over $\Fq$ along with equation \eqref{Sch} implies
\begin{equation}\label{schcard}
|C_\a| = q^{\delta (\a)} \ \ \ \text{and} \ \ \  |\Omega_\a| = \sum_{\beta \in \nabla (\a)} q^{\delta (\beta)}.
\end{equation}
From now on, we shall denote $|\Omega_\a|$ over $\Fq$ by $n_\a$.% for the cadinality of $\Omega_\a$ from now on. 
%references to be added

 %We conclude the section with illustrating the theory covered in this section through two examples. 

\begin{comment}
The Grassmann code $C(\ell, m)$ is the linear code corresponding to the projective system given by the points on the Grassmann variety $G (\ell, V_m)$ in the Pl\"ucker projective space. In particular, the length $n(\ell, m)$ and the dimension $k (\ell, m)$ of  $C (\ell, V_m)$ is given by
$$n (\ell, m) = {m \brack \ell}_q \ \ \ \text{and} \ \ \ k(\ell, m) = {m \choose \ell}.$$
The minimum distance of the code $C (\ell, m)$ was determined by Nogin \cite{N} in 1996. In fact, Nogin \cite[Theorem 2]{N} showed that the minimum distance of $C(\ell, m)$ is given by 
\begin{equation}\label{N}
d (\ell, m) = q^{\ell (m - \ell)}.
\end{equation}
In this paper, we give a new proof of \eqref{N}. Note that in order to compute $d (\ell, m)$, we will have to compute the following quantity:
\begin{equation}\label{elm}
e (\ell, m) = \max \{ |G (\ell, V_m) \cap \Pi|: \Pi \ \text{is a hyperplane of} \ \PP^{k-1} \},\end{equation}
where $\PP^{k-1}$ stands for the Pl\"ucker projective space $\PP(\bigwedge^\ell V_m)$. 
\end{comment}

\section{A decomposition of $G(\ell, V_{m})$} \label{dec}
As above, let $\ell, m$ be positive integers with $1 \le \ell \le m$. Let $V_m$ be an $m$-dimensional vector space over $\FF$. The object $G(m, V_m)$ being trivial, let us assume that $1 \le \ell \le m-1$. For ease of notation, let us denote by $\V_{m-1}$ the set of all subspaces of $V_m$ of dimension $m-1$.  
Let $W_{m-1} \in \V_{m-1}$, that is, $W_{m-1}$ is an arbitrary $(m-1)$-dimensional subspace of $V_m$. If $1 \le \ell \le m-1$, then the Grassmannian $G(\ell, W_{m-1})$ is naturally a subset of $G(\ell, V_m)$: Indeed, any $\ell$-dimensional subspace of $W_{m-1}$ is an $\ell$-dimensional subset of $V_m$. On the other hand, the Grassmannian $G(\ell, W_{m-1})$ can be identified as a Schubert variety in the following way. Fix a basis $\B' = \{v_1, \dots, v_{m-1} \}$ of $W_{m-1}$ and extend it to a basis $\B = \{v_1, \dots, v_m\}$ of $V_m$. We also fix an ordering of the elements of $I(\ell, m)$. The Grassmannian $G(\ell, W_{m-1})$ is, in fact, equal to the Schubert subvariety $\Omega_{\gamma}$, where $\gamma = (m- \ell, m - \ell+1,  \dots, m-1)$; that is, $\gamma$ consists of the top $\ell$ elements of $\{1, \dots, m-1\}$. 

Under the correspondence in \eqref{em}, an element $L \in G(\ell, V_m)$  belongs to $G(\ell, W_{m-1})$ if and only if the last column of  $M_{\mathcal{B}}(L)$ is zero, that is, the pivot in the last row of $M_{\mathcal{B}}(L)$ does not occur in the last column. Consequently, the elements of $G(\ell, V_m) \setminus G(\ell, W_{m-1})$ correspond exactly to matrices in $\mathcal{M}(\ell, m)$ whose last pivot occurs  in the last column, that is, those indexed by $\beta \in I(\ell, m)$ with $\beta_\ell = m$. As noted in the proof of Proposition \ref{prop:glm-1} below, this is equivalent to $\beta \not\leq \gamma$, 
that is, $\beta \in \Delta(\gamma)$, where $\gamma$ is as defined above. We have  
\begin{equation}\label{glm-1}
 G(\ell, W_{m-1}) = G(\ell, V_m) \cap V(\{X_{\beta} : \beta \in \Delta (\gamma)\}) \ \ \text{and} \ \    G(\ell, V_m) = G(\ell, W_{m-1}) \sqcup \bigsqcup_{\beta \in \Delta (\gamma)} C_\beta.
\end{equation}

We have the following Proposition. 

\begin{proposition}\label{prop:glm-1}
    Let $\Pi$ be a hyperplane of $\PP\left(\bigwedge^\ell V_{m} \right)$ and $W_{m-1}$ be an $(m-1)$-dimensional subspace of $V_m$. 
    \begin{enumerate}
        \item[(a)]  If $\Pi$ contains $G(\ell, W_{m-1})$, then there exists a basis $\B$ of $V_m$ and an ordering of elements in $I(\ell, m)$ such that $\Pi$ is given by an equation of the form $F(X_\a : \a \in I(\ell, m)) = 0$, where
        $$F(X_\a : \a \in I(\ell, m)) = \sum_{\substack{\a \in I(\ell, m) \\ \a_\ell = m}} c_\a X_\a = \sum_{\a \in \Delta (\gamma)} c_\a X_\a,$$ where $X_\a$ are the homogeneous coordinates of $\PP \left( \bigwedge^{\ell} V_m \right)$  defined with respect to $\B$. In particular, the variable $X_\a$ occurs in the defining polynomial of such a hyperplane if and only if $\a_\ell = m$.
        \item[(b)]  If $\Pi$ does not contain $G(\ell, W_{m-1})$, then $\Pi$ restricts to a hyperplane $\Pi_\gamma$ on $\PP^{k_\gamma - 1}$ and $\Pi \cap G(\ell, W_{m-1}) = \Pi_\gamma \cap G(\ell, W_{m-1})$. 
    \end{enumerate}
   \end{proposition}

\begin{proof} 
As usual, we choose an ordered basis $\B' = \{v_1, \dots, v_{m-1}\}$ of $W_{m-1}$ and extend it to an ordered basis $\B = \{v_1, \dots, v_m\}$ of $V_m$. We also fix an ordering of elements of $I(\ell, m)$. Let $\Pi$ be a hyperplane of $\PP\left(\bigwedge^\ell V_m \right)$ given by the polynomial $F(X_\a : \a \in I(\ell, m)) = \displaystyle{\sum_{\a \in I(\ell, m)}} c_\a X_\a$ with respect to the basis $\B$ of $V_m$. In view of Remark \ref{rem1}, we see that $\Pi$  contains $G(\ell, W_{m-1})$ if and only if $c_\beta = 0$ for all $\beta \in \nabla (\gamma)$. Thus 
$$F(X_\a : \a \in I(\ell, m)) = \sum_{\a \in \Delta (\gamma)} c_\a X_\a.$$
This proves part (a) observing that $\beta \in \Delta (\gamma)$ if and only if $\beta_\ell = m$, as noted above. For (b), Remark \ref{rem1} implies that indeed $\Pi$ restricts to a hyperplane $\Pi_\gamma$ in $\PP^{k_{\gamma} - 1}$. Clearly, $\PP^{k_{\gamma} - 1}$ is equal to $\PP \left(\bigwedge^{\ell} W_{m-1} \right)$ and 
$$G(\ell, W_{m-1}) \cap \Pi = G(\ell, W_{m-1}) \cap \PP^{k_{\gamma} -1} \cap \Pi = G(\ell, W_{m-1}) \cap \Pi_\gamma.$$
This completes the proof. 
\end{proof}
As before, let $W_{m-1}$ denote an $(m-1)$-dimensional subspace of $V_m$.  We continue with the fixed ordered basis  $\mathcal{B}= \{v_1, \dots, v_m\}$ of $V_m$ and the fixed ordering of the elements in $I(\ell, m)$. Under the correspondence in \eqref{em}, the elements of  $G(\ell, W_{m-1})$ correspond to matrices in $\M (\ell, m)$ whose last column is zero, and hence, the elements of $G(\ell, V_{m}) \setminus G(\ell, W_{m-1})$ correspond to matrices in  
$$T(\ell, m) := \{M \in \M(\ell, m) : \ \text{the pivot in the last row of} \ M \ \text{occurs in the last column}\}.$$ 
Let $M \in T(\ell, m)$.  We may write
\begin{equation}\label{cm}
    M = \begin{pmatrix}
        \widecheck{M} &&  {\bf 0} \\ 
          c_M && 1
    \end{pmatrix},
\end{equation}
where $\widecheck{M} \in \M(\ell - 1, m-1)$, the symbol ${\bf 0}$ denotes a column matrix of length $(\ell - 1)$ with all entries equal to ${0}$, whereas $c_M$ is a row vector of length $m-1$, and of course $1 \in \Fq$.  Indeed, since $M$ is in right-row-reduced echelon form, so is $\widecheck{M}$, which justifies the inclusion $\widecheck{M} \in \M(\ell - 1, m-1)$. The pivot columns of $\widecheck{M}$ are determined by those of $M$, and the entries in $c_M$ on the pivot columns of $\widecheck{M}$ are zero. 
 
Let us denote $(M_{n_1}, \dots, M_{n_{m-\ell}})$ the entries in the non-pivotal columns of $c_M$, where $n_1 < \cdots < n_{m-\ell}$. This defines a map 
$$s : T (\ell, m) \to \FF^{m-\ell}.$$
The map $s$ is clearly surjective. We claim that $s^{-1} (\nu)$ can be identified with $\M (\ell- 1, m-1)$ for any $\nu = (\nu_1, \dots, \nu_{m-\ell})  \in \FF^{m-\ell}$. 
Define 
\begin{equation}\label{bij2}
\phi_\nu : \M (\ell-1, m-1) \to s^{-1} (\nu) \ \ \ \text{by} \ \ \ M' \mapsto \begin{pmatrix}
        {M'} &&  0 \\ 
          c_{M', \nu} && 1 \end{pmatrix},\end{equation}
where $c_{M', \nu}$ is uniquely determined by the non-pivot positions of $M'$ and $\nu$. Clearly, the map $\phi_\nu$ is a bijection. Since the preimages of distinct elements are disjoint, we have
\begin{equation}\label{t}
  G(\ell, V_m) = G(\ell, W_{m-1}) \bigsqcup T(\ell, m)  \ \ \ \text{and} \ \ \         T(\ell, m)  = \bigsqcup_{\nu \in \FF^{m-\ell}} s^{-1} (\nu).  
\end{equation}

We call the fibre $s^{-1} (\nu)$  \emph{the $\nu$-th string\footnote{An absolutely unimportant but fun remark: We use the term ``strings'' as an informal metaphor: think of $G(\ell,m)$ as a guitar, with $G(\ell,W_{m-1})$ forming the body and $G(\ell,m) \setminus G(\ell,W_{m-1})$ forming the fretboard. The Schubert cells act like frets, while our decomposition in (24) slices the fretboard into ``strings.'' } of $V_m$ with respect to $W_{m-1}$} and the basis $\mathcal{B}$ of $V_m$. %Now we are ready to state and prove the main theorem of this article. 
While it is not evident whether the strings qualify as (quasi)projective subvarieties of $G(\ell, V_m)$, we can still identify them with Grassmannians $G(\ell - 1, W_{m-1})$ via the bijections \eqref{em} and \eqref{bij2}. 

Let us illustrate a quick application of the above discussion in the context of hyperplane sections of the Grassmann variety $G(\ell, V_m)$. We keep working with an $(m-1)$-dimensional subspace $W_{m-1}$ of $V_m$ as above. To this end, let $\Pi$ be a hyperplane of $\PP(\bigwedge^\ell V_m)$ containing $G(\ell, W_{m-1})$. In view of the decomposition in \eqref{t}, it follows that
\begin{equation}\label{section}
    G(\ell, V_m) \cap \Pi = G(\ell, W_{m-1}) \bigsqcup \left( \bigsqcup_{\nu \in \FF^{m-\ell}} \left(s^{-1} (\nu) \cap \Pi\right)\right)
\end{equation}
From Proposition \ref{prop:glm-1}, we see that a homogeneous linear polynomial $F(X_\a : \a \in I(\ell, m))$ representing $\Pi$ is given by 
\begin{equation}\label{special}
F(X_\a : \a \in I(\ell, m)) = \sum_{\substack{\a \in I(\ell, m) \\ \a_\ell = m}} c_\a X_\a.
\end{equation}
Note that, under our setting, there is a natural \textit{projection} map
$\tau : G(\ell, V_m) \setminus G(\ell, W_{m-1}) \to G(\ell - 1, W_{m-1})$ given by the composition of the maps:
$$G(\ell, V_m) \setminus G(\ell, W_{m-1}) \longleftrightarrow T(\ell, m) \to \M(\ell - 1, m-1) \longleftrightarrow G(\ell - 1, W_{m-1}),$$
where the map $T(\ell, m) \to \M(\ell - 1, m-1)$ is given by $M \mapsto \widecheck{M}$. It follows that for every $\a = (\a_1, \dots, \a_\ell) \in I(\ell, m)$ with $\a_\ell = m$ and $L \in T(\ell, m)$, we have $X_\a (L) = X_{\widecheck{\a}} (\tau (L))$, where $\widecheck{\alpha} = (\a_1, \dots, \a_{\ell - 1}) \in I(\ell - 1, m-1)$. 
Thus
\begin{equation}\label{eval}
F(X_\a : \a \in I(\ell, m)) (L) = \sum_{\substack{\a \in I(\ell, m) \\ \a_\ell = m}} c_\a X_{\widecheck{\a}}(\tau(L)).
\end{equation}
Given a homogeneous linear polynomial $F(X_\a : \a \in I(\ell, m))$ as in \eqref{special}, let us denote by $\widecheck{F} (X_{\widecheck{\a}} : {\widecheck{\a}} \in I(\ell-1, m-1))$ the polynomial
$$\widecheck{F} (X_{\widecheck{\a}} : {\widecheck{\a}} \in I(\ell-1, m-1)) = \sum_{\substack{\a \in I(\ell, m) \\ \a_\ell = m}} c_\a X_{\widecheck{\a}}.$$
Thus for any hyperplane $\Pi$  in $\PP (\bigwedge^{\ell} V_m)$ containing $G(\ell, W_{m-1})$ given by the polynomial $F$ in \eqref{special}, we get a hyperplane $\widecheck{\Pi}$ of $\PP (\bigwedge^{\ell - 1} W_{m-1})$ given by $\widecheck{F}$. In view of \eqref{eval}, we get a one-to-one correspondence
\begin{equation}\label{check}
s^{-1} (\nu) \cap \Pi \leftrightarrow G(\ell - 1, W_{m-1}) \cap \widecheck{\Pi}.
\end{equation}
If $\FF = \Fq$, then it also follows from \eqref{check} that for all $\nu \in \FF^{m-\ell}$, the sets $s^{-1} (\nu) \cap \Pi$ have the same number of elements.  
We conclude this section with the following elementary observations. 

\begin{remark}\label{rem2}\normalfont
Suppose $F(X_\a : \a \in I(\ell, m))$ is a nonzero homogeneous linear polynomial in $X_\a$ with $\a_\ell = m$. As noted above, this is equivalent to the statement that the hyperplane $\Pi$ of $\PP (\bigwedge^{\ell} V_m)$ defined by the equation $F = 0$ contains $G(\ell, W_{m-1})$. Let $\widecheck{\Pi}$ and $\widecheck{F}$ be as above. 
    \begin{enumerate}
        \item[(a)] Suppose there exists $\nu \in \FF^{m-\ell}$ such that $F$ vanishes at all points of $s^{-1}(\nu)$. From \eqref{eval}, we see that $\widecheck{F}$ vanishes at all points of $G(\ell - 1, W_{m-1})$. Since $G(\ell - 1, W_{m-1})$ is not contained in any hyperplanes of $\PP(\bigwedge^{\ell-1} W_{m-1})$, it follows that $\widecheck{F}$ is the zero polynomial. Since the 
coefficients of $F$ and $\widecheck{F}$ are the same, we conclude 
that $F$ is also the zero polynomial, contradicting our assumption.

        \item[(b)] As discussed in the previous section, the polynomial $F(X_\a : \a \in I(\ell, m)) = \displaystyle{\sum_{\a \in I(\ell, m)}} c_\a X_\a$ identifies with the element $\displaystyle{\sum_{\a \in I(\ell, m)}} c_\a \epsilon (\a) v_{\a^{\mathsf{C}}} \in \PP \left(\bigwedge^{m-\ell} V_m\right)$. It is easily verified that $\widecheck{F}$ identifies with exactly the same element of $\PP (\bigwedge^{m-\ell} V_m)$. In particular, $\Pi$ is a decomposable hyperplane in $\PP (\bigwedge^{\ell} V_m)$ if and only if $\widecheck{\Pi}$ is a decomposable hyperplane in $\PP \left(\bigwedge^{\ell-1} W_{m-1}\right)$. 
    \end{enumerate}
\end{remark}

\section{The minimum distance and the second minimum distance of Grassmann codes}\label{sec:main}
In this section, we present our main result on the second minimum distance of Grassmann codes. We first take a step back to revisit Nogin's Theorem on the minimum distance of Grassmann codes and give an independent proof of the same. Throughout, we will assume that $m$ is a positive integer and $V_m$ is a vector space of dimension $m$ over a finite field $\Fq$. 
We begin with the following Lemma motivated by the so-called Zanella's lemma often used in the context of projective Reed-Muller codes. 
\begin{lemma}\label{zan}
    Let $\Pi$ be a hyperplane in $\PP\left(\bigwedge^{\ell} V_m \right)$. If $a = \displaystyle{\max_{W_{m-1} \in \V_{m-1}}} |\Pi \cap G (\ell, W_{m-1})|$, then $$|G(\ell, V_m) \cap \Pi| \le a \left(\frac{q^ m - 1}{q^{m - \ell} - 1}\right).$$
    Furthermore, if $a = |\Pi \cap G (\ell, W_{m-1})|$, for every $W_{m-1} \in \V_{m-1}$, then equality holds.
\end{lemma}

\begin{proof}
    Consider the incidence set,
    $$\mathscr{I} = \{(L, W_{m-1}) : L \in G(\ell, V_{m}), \  W_{m-1} \in \V_{m-1}, \ L \in G(\ell, W_{m-1}) \cap \Pi \}.$$
    We count $|\mathscr{I}|$ in two ways. First note that,
    \begin{align}\label{i1}
        |\mathscr{I}| &= \sum_{L \in G(\ell, V_m) \cap \Pi} |\{W_{m-1} \in \V_{m-1}: L \subset W_{m-1} \subset V_m\}| \nonumber  \\
        &= |G(\ell, V_m) \cap \Pi| {m-\ell \brack m - \ell - 1}_q \nonumber \\
        &= |G(\ell, V_m) \cap \Pi| \left(\frac{q^{m - \ell} - 1}{q-1}\right). 
    \end{align}
Indeed the last equality follows Proposition \ref{identities} (a). 
    On the other hand, 
    \begin{equation}\label{i2}
        |\mathscr{I}| = \sum_{W_{m-1} \in \V_{m-1}} |\Pi \cap G(\ell, W_{m-1})| \le \sum_{{W_{m-1} \in \V_{m-1}}} a = a \left(\frac{q^m - 1}{q - 1}\right).
    \end{equation}
    The first assertion now follows trivially from \eqref{i1} and \eqref{i2}. The second assertion is trivial. 
\end{proof}

Let us record some combinatorial (in)equalities in the following proposition for the sake of ease of reference. First, for $1 \le \ell \le m$, we denote by
$$e (\ell, m): = {m \brack \ell}_q - q^{\ell (m - \ell)} \ \ \ \text{and} \ \ \ e' (\ell, m): = {m \brack \ell}_q - q^{\ell (m - \ell)} - q^{\ell (m - \ell) - 2}.$$

\begin{proposition}\label{ineq}
    Let $\ell, m$ be positive integers with $1 \le \ell \le m-1$. We have 
    \begin{enumerate}
        \item[(a)] $e (\ell, m-1) \left(\frac{q^m -1}{q^{m-\ell} - 1}\right) < e (\ell, m).$
        \item[(b)] ${m - 1 \brack \ell}_q + q^{m-\ell} e (\ell -1, m-1) = e (\ell, m). $
        \item[(c)] If $2 \le \ell \le m-2$, then $e' (\ell, m-1) \left(\frac{q^m -1}{q^{m-\ell} - 1}\right) < e' (\ell, m).$
        \item[(d)] If $2 \le \ell \le m-2$, then ${m - 1 \brack \ell}_q + q^{m-\ell} e' (\ell -1, m-1) = e' (\ell, m)$.
    \end{enumerate}
\end{proposition}

\begin{proof}
    \
    \begin{enumerate}
        \item[(a)] Note that \begin{align*}
        e (\ell, m-1) \left(\frac{q^m -1}{q^{m-\ell} - 1}\right) &=\left(\frac{q^m - 1}{q^{m-\ell} - 1} \right) \left( {m - 1 \brack \ell}_q - q^{\ell(m-1 - \ell)} \right) \\
        &= {m \brack \ell}_q - \left(\frac{q^m - 1}{q^{m-\ell} - 1}\right) q^{\ell(m-1 - \ell)} \\
        &< {m \brack \ell}_q - q^{\ell (m-\ell)} = e(\ell, m).
    \end{align*}
    The equality above follows from Proposition \ref{identities} (c). 
    while the inequality on the last line is equivalent to the inequality $(q^m - 1) > q^\ell (q^{m - \ell} - 1) = q^m - q^\ell$, which is trivially true as $\ell \ge 1$ and $q$ is a prime power. 
    \item[(b)] The left-hand side of the assertion is given by
        \begin{equation*}
        {m - 1 \brack \ell}_q + q^{m - \ell} \left({ m - 1 \brack \ell - 1}_q - q^{(\ell-1)(m - \ell)} \right) \\
       =  {m \brack \ell}_q - q^{\ell(m- \ell)} = e(\ell, m). 
    \end{equation*} 
    The first equality follows from  Proposition \ref{identities} (b). 
    \end{enumerate}
    We leave the proof of parts (c) and (d), which can be easily deduced in a way similar to parts (a) and (b) respectively, to the reader. 
\end{proof}

We are now ready to restate and prove Nogin's Theorem. 
\begin{theorem}[Nogin] \cite[Theorem 2]{N} \label{main}
Let $1 \le \ell \le m$ be positive integers. If $\Pi$ is a hyperplane in $\PP(\bigwedge^{\ell} V_m)$, then 
$|\Pi \cap G(\ell, V_m)| \le e(\ell, m)$. 
The equality holds if and only if  $\Pi$ is decomposable. 
\end{theorem}

\begin{proof}
    We will prove the assertion by induction on $m$. The assertions are trivial when $m = 1$. Thus we may assume that $m > 1$ and that the assertions are true for $G(\ell, W_{m-1})$ for every $W_{m-1} \in \V_{m-1}$.
    Note that, if $\ell = m$, then the assertion is trivial. Thus we may assume $1 \le \ell < m$. Let $\Pi$ be a hyperplane $\PP(\bigwedge^{\ell} V_m)$. We distinguish the proof into two cases. 

    \textbf{Case 1:} \emph{Suppose $\Pi$ does not contain $G(\ell, W_{m-1})$ for any $W_{m-1} \in \V_{m-1}$}. From induction hypothesis and Proposition \ref{prop:glm-1} (b), for any $W_{m-1} \in \V_{m-1}$, we have $|\Pi \cap G(\ell, W_{m-1})| = |\Pi_\gamma \cap G(\ell, W_{m-1})| \le e(\ell, m-1)$.
    Lemma \ref{zan} implies
    \begin{equation}
        |G(\ell, V_m) \cap \Pi| \le \left(\frac{q^m - 1}{q^{m-\ell} - 1} \right) e(\ell, m-1) < e(\ell, m) .
    \end{equation}
    The last equality follows from Proposition \ref{ineq} (a). 

    \textbf{Case 2:} \emph{There exists  $W_{m-1} \in \V_{m-1}$ such that $\Pi$ contains $G(\ell, W_{m-1})$}. From \eqref{section}, we have
    \begin{align*}
        |G(\ell, V_m) \cap \Pi| &= |G(\ell, W_{m-1})| + \sum_{\nu \in \Fq^{m-\ell}} |s^{-1} (\nu) \cap \Pi| \\
        & = {m - 1 \brack \ell}_q + \sum_{\nu \in \Fq^{m-\ell}} |G(\ell - 1, W_{m-1}) \cap \widecheck{\Pi}| \\
        & \le {m - 1 \brack \ell}_q + q^{m - \ell}e (\ell - 1, m-1) \\
       &=  e (\ell, m). 
    \end{align*}
 The second equality follows from the bijection in \eqref{check}, the inequality from the induction hypothesis, and the last equality is a consequence of Proposition \ref{ineq} (b).  It is clear from the proof, that a hyperplane $\Pi$ of $\PP (\bigwedge^{\ell} V_m)$ attains the bound if and only if $\Pi$ contains $G(\ell, W_{m-1})$ and $|\Pi \cap s^{-1} (\nu)| = e(\ell - 1, m-1)$ for all $\nu \in \Fq^{m-\ell}$. Thus $|G(\ell - 1, W_{m-1}) \cap \widecheck{\Pi}| = e(\ell - 1, m-1)$. By induction hypothesis, $\widecheck{\Pi}$ is decomposable. From Remark \ref{rem2} (b), we conclude that $\Pi$ is decomposable. This completes the proof. 
\end{proof}
Before proceeding further let us add a remark on the number of minimum weight codewords of the Grassmann codes. 

\begin{remark}\normalfont
As noted in the proof of Theorem~\ref{main}, a hyperplane $\Pi$ achieves the minimum weight of $C(\ell, m)$ if and only if 
$\Pi$ is decomposable. The set of decomposable hyperplanes in $\mathbb{P}(\bigwedge^\ell V_m)$ is in natural bijection with 
$G(m-\ell, V_m)$. Since each hyperplane gives rise to exactly $(q-1)$ codewords, we see that the number of minimum weight codewords of $C(\ell, m)$ over $\mathbb{F}_q$ equals
$$
|G(m-\ell, V_m)|(q-1) =  \binom{m}{\ell}_q (q-1).
$$
This was already established by Nogin in \cite[Corollary 4.5]{N}.
\end{remark}

We have the following Corollary.

\begin{corollary}\label{cor:glm-1}
    Let $\Pi$ be a hyperplane of $\PP\left(\bigwedge^\ell V_{m} \right)$ and $W_{m-1} \in \V_{m-1}$. We have: 
    \begin{enumerate}
        \item[(a)] If $\Pi$ does not contain $G(\ell, W_{m-1})$, then $$|G(\ell, W_{m-1}) \cap \Pi| \le  {m - 1 \brack \ell}_q - q^{\ell (m-1-\ell)}.$$
        \item[(b)] If $|\Pi \cap G(\ell, W_{m-1})| = {m - 1 \brack \ell}_q - q^{\ell (m-1-\ell)},$ for some $W_{m-1} \in \V_{m-1}$, then $\Pi$ is given by a homogeneous polynomial of the form $\displaystyle{\sum_{\a \in \Delta (\gamma)}} c_\a X_\a + X_\gamma$ with respect to a basis $\B$ of $V_m$, where $\gamma = (m - \ell, \dots, m - 1)$ as defined before. 
    \end{enumerate}
\end{corollary}

\begin{proof}
\
\begin{enumerate}
    \item[(a)] Follows from Proposition \ref{prop:glm-1} and Nogin's Theorem. 
    \item[(b)] From Proposition \ref{prop:glm-1} (b), the hyperplane $\Pi$ restricts to a hyperplane $\Pi_{\gamma}$ on $\PP^{k_{\gamma} - 1}$. It follows from Nogin's Theorem on the classification of minimum weight codewords of $C(\ell, m-1)$ that $\Pi_\gamma$ is decomposable. Thus there exist linearly independent elements $w_1, \dots, w_{m-1-\ell} \in W_{m-1}$ such that $\Pi_\gamma$ is given by $F_\gamma = w_1 \wedge \cdots \wedge w_{m-1-\ell}$. We extend $\{w_1, \dots, w_{m-1-\ell}\}$ to an ordered basis ${\B'} =\{w_1, \dots, w_{m-1}\}$ of $W_{m-1}$ and ${\B'}$ to an ordered basis ${\B} = \{w_1, \dots, w_m\}$ of $V_m$. We redefine the coordinates of $\PP\left(\bigwedge^\ell W_{m-1} \right)$ and $\PP\left(\bigwedge^\ell V_{m} \right)$ with respect to the basis ${\B}$ and write the equation of $\Pi$ as $F(X_\a : \a \in I(\ell, m)) = \sum_{\a \in I(\ell, m)} c_\a X_\a$. Since, in terms of wedge products, the hyperplane $\Pi_\gamma$ is given by $w_1 \wedge \dots \wedge w_{m-1-\ell}$, from the correspondence in \eqref{hypid}, we see that $\Pi_\gamma$ is given by the equation $\epsilon(\gamma) F_\gamma = 0$. The correspondence in \eqref{hypid} shows that $\epsilon(\gamma) F_{\gamma} = X_\gamma$. Thus $F = \displaystyle{\sum_{\a \in \Delta (\gamma)}} c_\a X_\a + X_\gamma$.
\end{enumerate}
This completes the proof. 
\end{proof}

We now proceed with our investigation towards determination of the second minimum weight of the Grassmann code $C(\ell, m)$. We observe that, if $\ell = 1$ or $\ell = m - 1$, then the Grassmannian $G(\ell, V_m)$ is a projective space $\PP^{m-1}$. Consequently, any hyperplane intersects it at exactly $p_{m-2}$ many points, where $p_{m-2} = \frac{q^{m-1} - 1}{q-1}$. Thus, we may assume that $2 \le \ell \le m - 2$. First, we have to refer to a very special Schubert subvariety of $G(\ell, V_m)$.
We continue with the fixed vector space $V_m$ of dimension $m$ over $\Fq$ with an ordered basis $\B = \{v_1, \dots, v_m\}$ and an ordering of the elements in $I(\ell, m)$, as in the beginning of this section. Consider the Schubert variety $\Omega_\theta$ with respect to the basis $\B$, where $\theta = (m - \ell -1, m - \ell + 2, m-\ell + 3, \dots, m)$, consisting of $m-\ell -1$ together with the top $\ell - 1$ integers in $\{1, \dots, m\}$. This particular $\theta$ is chosen because it is the unique 
element of $I(\ell, m)$ satisfying $|\Delta(\theta)| = \ell + 1$ and $\delta(\theta) = \ell(m - \ell) - 2$, which as we shall 
see is precisely what is needed to establish the second minimum weight formula. Let us first describe the subset $\Delta (\theta)$. Since for every $\beta = (\beta_1, \dots, \beta_\ell) \in I(\ell, m)$, we have $\beta_{i} \le m - \ell + i$, for  $i = 1, \dots, \ell$, we must have $\beta_1 \ge m- \ell$ for every $\beta \in \Delta (\theta)$. Thus $\Delta (\theta)$ consists of exactly $\ell + 1$ elements, denoted by $\theta_1, \dots, \theta_{\ell + 1}$, that are obtained by writing down the $(\ell + 1)$-tuple $(m - \ell, m - \ell + 1, \dots, m)$, and dropping the $i$-th entry for each $i = 1, \dots, \ell + 1$.
In particular, we have $\theta_1 = (m-\ell + 1, \dots, m)$ and  $\theta_{\ell + 1} = (m - \ell, \dots, m-1) = \gamma$, whereas 
$$\theta_i = (m - \ell, \dots, m - \ell + i - 2, m - \ell + i, \dots, m) \ \ \ \text{for each} \ \ \ i = 2, \dots, \ell.$$
From \eqref{commd1} and \eqref{Sch} we have,
\begin{equation}\label{thesch}
G(\ell, V_m) = \Omega_\theta \sqcup \bigsqcup_{i = 1}^{\ell + 1} C_{\theta_i} \ \ \ \text{and} \ \ \ \Omega_\theta = G(\ell, V_m) \cap V(\{X_{\theta_i} : i = 1, \dots, \ell+1\}).
\end{equation}
  We have the following Proposition. 
\begin{proposition}\label{schspec}
    Suppose $\Pi$ is a hyperplane of $\PP\left(\bigwedge^\ell V_{m} \right)$. Let $\Omega_\theta$ be as described above.
    \begin{enumerate}
        \item[(a)]  If $\Pi$ contains $\Omega_{\theta}$, then $\Pi$ is decomposable.  
        \item[(b)] If $\Pi$ does not contain $\Omega_\theta$, then $|\Pi \cap \Omega_\theta| \le n_\theta - q^{\delta(\theta)}$. 
    \end{enumerate}
\end{proposition}

\begin{proof}
    \
    \begin{enumerate}
        \item[(a)] Recall from Remark~\ref{rem1} that a hyperplane $\Pi$ of $\PP\left(\bigwedge^{\ell} V_m\right)$ contains $\Omega_\theta$ if and only if in the defining polynomial the coefficients of $X_\beta$  vanish  for all $\beta \in \nabla(\theta)$, that is, the only nonzero 
coefficients are those corresponding to $\beta \in \Delta(\theta)$. 
Since $\Delta(\theta) = \{\theta_1, \ldots, \theta_{\ell+1}\}$, 
denoting the coefficient of $X_{\theta_i}$ by $c_i$, the 
equation of such a hyperplane takes the simple form 
$\sum_{i=1}^{\ell+1} c_i X_{\theta_i} = 0$.  In view of \eqref{hypid}, the defining element of $\Pi$ in $\bigwedge^{m-\ell} V_m$ is given by 
$$\sum_{i=1}^{\ell+1} \epsilon(\theta_i) c_i v_{\theta_i^{\mathsf{C}}} = v_1 \wedge \dots \wedge v_{m-\ell - 1} \wedge \left(\sum_{i = m-\ell}^m \epsilon(\theta_i) c_i v_i\right),$$
which is a decomposable element of $\bigwedge^{m - \ell} V_{m}$. 
\item[(b)] Take a hyperplane $\Pi$ of $\PP\left(\bigwedge^\ell V_{m} \right)$ given by the equation $\displaystyle{\sum_{\a \in I(\ell, m)}} c_\a X_\a = 0$ that does not contain $\Omega_\theta$. As noted in Remark \ref{rem1}, the hyperplane $\Pi$ restricts to a hyperplane $\Pi_\theta$ on $\PP^{k_\theta - 1}$. The inequality \eqref{minsch} implies 
$$  
    |\Pi \cap \Omega_\theta| = |\Pi_\theta \cap \Omega_\theta| \le n_{\theta} - q^{\delta (\theta)}.
$$
    \end{enumerate}
    This completes the proof.
\end{proof}

The quantities $n_{\theta}$ and $\delta ({\theta})$ can be computed  easily. First, from the description of $\theta_i$'s, we have 
\begin{equation}\label{deltheta}
\delta (\theta_i) = \sum_{j=0}^{\ell} (m- \ell + j) - (m - \ell + i - 1) - \frac{\ell(\ell + 1)}{2} = \ell (m - \ell) - (i-1),
\end{equation}
for all $i= 1, \dots, \ell + 1$. Consequently, we have, 
\begin{equation}\label{ntheta}
n_\theta = {m \brack \ell}_q - \sum_{i=1}^{\ell + 1} q^{\delta (\theta_i)} =  {m \brack \ell}_q - \sum_{i=1}^{\ell + 1} q^{\ell (m - \ell) - (i - 1)} = {m \brack \ell}_q - q^{\ell(m-\ell)} \sum_{i=1}^{\ell + 1} \frac{1}{q^{i-1}}.
\end{equation}
Furthermore, a simple calculation shows that $\delta (\theta) = \ell (m - \ell) - 2$.

\begin{proposition}\label{submax}
    Let $\ell, m$ be positive integers satisfying $2 \le \ell \le m-2$ and $\Pi$ be a nondecomposable hyperplane of $\PP\left(\bigwedge^\ell V_m\right)$ such that $|G(\ell, W_{m-1}) \cap \Pi|= e(\ell, m-1)$ for some  $W_{m-1} \in \V_{m-1}$. Then $|G(\ell, V_m) \cap \Pi| \le e' (\ell, m).$
\end{proposition}

\begin{proof}
    As observed in Corollary \ref{cor:glm-1}, there exists an ordered basis $\B = \{v_1, \dots, v_m\}$ of $V_m$, such that the hyperplane $\Pi$ is given by a linear polynomial $$F(X_\a : \a \in I(\ell, m)) = \sum_{\a \in \Delta (\gamma)} c_\a X_\a + X_\gamma, \ \ \text{where} \ \ \gamma = (m-\ell, \dots, m-1).$$ Note that it is essential to have the coefficient of $X_\gamma$ to be nonzero, as otherwise, the polynomial $F$ is decomposable. 
    
    Since $\Pi$ is not decomposable, it follows from the discussion in Proposition \ref{schspec} (a), that $\Pi$ does not contain the Schubert subvariety $\Omega_\theta$. From Proposition \ref{schspec} (b), we have 
    \begin{equation}\label{theta1}
    |\Pi \cap \Omega_\theta| \le n_\theta - q^{\ell (m-\ell) - 2}. 
    \end{equation}
    We claim  that $\Pi$ does not intersect $C_{\theta_{\ell+1}}$. First note that for all $\a \in \Delta (\gamma)$, the polynomial $X_\a$ vanish identically on $C_{\theta_{\ell + 1}}$. Thus the polynomial $F$ restricts to $X_\gamma$ on $C_{\theta_{\ell + 1}}$, which does not have any zeroes in that cell. Now fix $k \in \{1, \dots, \ell\}$ and an element $L \in C_{\theta_k}$. Write $M_\B (L) = (a_{ij})$ as the $\ell \times m$ matrix in the right-row-reduced-echelon form (as defined in subsection \ref{sec:gra}) representing $L$ with respect to the basis $\B.$ It follows that 
    $p_\gamma (M_\B (L))$ is either $a_{\ell, m-\ell + k - 1}$ or $-a_{\ell, m-\ell + k - 1}$. On the other hand, we see that  $\sum_{\a \in \Delta (\gamma)} c_\a p_\a (M_\B (L))$ does not involve any of the entries from the $\ell$-th row of $M_\B (L)$. Consequently, if $F( M_\B (L)) = 0$, then $a_{\ell, m-\ell + k - 1}$ is uniquely determined by all the other entries in $M_{\B} (L)$ by means of the equation. Since $C_{\theta_k}$-s are affine spaces of dimensions $\delta (\theta_k)$ for $k = 1, \dots, \ell$, we have
    \begin{equation}\label{theta2}
    |\Pi \cap C_{\theta_k}| = q^{\delta (\theta_k) - 1} \ \ \ \text{for all} \ \ k = 1, \dots, \ell.
    \end{equation}
    Combining \eqref{thesch}, \eqref{deltheta}, \eqref{ntheta}, \eqref{theta1}, and \eqref{theta2}, we have
    \begin{align*}
    |\Pi \cap G(\ell, V_m)| &= |\Pi \cap \Omega_\theta| + \sum_{k=1}^\ell |\Pi \cap C_{\theta_k}|  + |\Pi \cap C_{\theta_{\ell+1}}| \\
    &\le n_\theta - q^{\ell (m-\ell) - 2} + \sum_{i=1}^\ell q^{\delta (\theta_k) - 1} \\
    &= \left({m \brack \ell}_q - \sum_{k=1}^{\ell + 1} q^{\delta (\theta_k)} \right) - q^{\ell (m-\ell) - 2} + \sum_{k=1}^{\ell} q^{\delta (\theta_k) - 1} \\
    &= \left({m \brack \ell}_q - q^{\ell(m-\ell)} \sum_{k=1}^{\ell + 1} \frac{1}{q^{k-1}}\right) - q^{\ell (m-\ell) - 2} + q^{\ell(m-\ell)} \sum_{k=1}^{\ell} \frac{1}{q^{k}} \\
    &= {m \brack \ell}_q - q^{\ell (m - \ell)} - q^{\ell  (m - \ell) - 2}.
   \end{align*}
   This completes the proof.
    \end{proof}

\begin{remark}\label{rem:attained}\normalfont
It is clear from the proof of Proposition~\ref{submax} that 
the bound is attained by a hyperplane $\Pi$ if and only if 
$|\Pi \cap \Omega_\theta| = n_\theta - q^{\ell(m-\ell)-2}$.
While a complete classification of all hyperplanes attaining 
this bound is not known, we exhibit an explicit class. Consider 
hyperplanes of the form 
$$c_\theta X_\theta + \sum_{\a \in \Delta(\theta) \setminus 
\{\gamma\}} c_\a X_\a + X_\gamma = 0, \ \ \text{with} \ \ 
c_\theta \neq 0.$$
Since $\gamma \in \Delta(\theta)$, the polynomial 
$\sum_{\a \in \Delta(\theta) \setminus \{\gamma\}} c_\a X_\a 
+ X_\gamma$ vanishes identically on $\Omega_\theta$. Thus the 
restriction of $\Pi$ to $\Omega_\theta$ is given by 
$c_\theta X_\theta = 0$. Since $c_\theta \neq 0$, this defines 
a hyperplane section of $\Omega_\theta$, given by $X_\theta = 0$ 
in $\PP^{k_\theta - 1}$. Note that, for any $L \in \Omega_\theta$, 
$X_{\theta} (L) \neq 0$ if and only if $L \in C_{\theta}$. Thus, 
there are exactly $|C_\theta| = q^{\delta(\theta)}$ many points 
in $\Omega_\theta$ where $X_\theta$ does not vanish. 
Consequently, the hyperplane $\Pi$ intersects $\Omega_\theta$ 
at exactly $n_\theta - q^{\delta(\theta)} = n_\theta - 
q^{\ell(m-\ell)-2}$ many points. In particular, the bound in 
Proposition~\ref{submax} is optimal.
\end{remark}

\begin{corollary}\label{cor}
Let $\Pi$ be a nondecomposable hyperplane of $\PP\left(\bigwedge^{m-2} V_m\right)$. Then $$|G(m - 2, W_{m-1}) \cap \Pi|= e(m-2, m-1) \ \ \text{for all} \ \   W_{m-1} \in \V_{m-1}.$$ Consequently, we have $|G(m-2, V_m) \cap \Pi| \le e'(m-2, m).$
\end{corollary}

\begin{proof}
   Let $W_{m-1} \in \V_{m-1}$. Note that $G(m-2, W_{m-1})$ is a projective space of dimension $m-2$. Consequently, if $\Pi$ is a hyperplane of $\PP \left(\bigwedge^{m-2} V_m \right)$, then either $\Pi$ contains $G(m-2, W_{m-1})$ or $\Pi$ intersects $G(m-2, W_{m-1})$ at a projective linear subspace of codimension $1$. That is, if $\Pi$ does not contain $G(m-2, W_{m-1})$, then $|\Pi \cap G(m-2, W_{m-1})| = \frac{q^{m-2} - 1}{q-1}$. On the other hand, $e (m-2, m-1) = {m-1 \brack m-2}_q - q^{m-2} = \frac{q^{m-2} - 1}{q-1}$. It is thus enough to show that $\Pi$ does not contain $G(m-2, W_{m-1})$. As usual, we fix a basis $\B' =\{v_1, \dots, v_{m-1}\}$ of $W_{m-1}$ and extend it to a basis $\B = \{v_1, \dots, v_m\}$ of $V_m$. With respect to this basis, we see that any hyperplane $\Pi$ of $\PP \left(\bigwedge^{m-2} V_m \right)$ is given by a polynomial 
   $$F (X_\a : \a \in I(m-2, m)) = \sum_{\a_\ell = m}{c_\a X_\a}.$$
   But it is easy to see that $F$ is decomposable which contradicts the hypothesis. 
\end{proof}

In the case when $\ell = 2$, the complete weight distribution of the Grassmann code $C(2, m)$ was determined by Nogin. In particular, the second minimum weight is already known. However, for the sake of completeness, we give an independent proof for the second minimum distance of $C(2, m)$. This is equivalent to the following Proposition. 

\begin{proposition}\label{l2}
    Let $\Pi$ be a hyperplane of $\PP\left(\bigwedge^2 V_m\right)$ such that $|\Pi \cap G(2, V_m)| < e (2, m)$. Then
    $|\Pi \cap G(2, V_m)| \le e' (2, m).$
    Moreover, the above bound is attained.
\end{proposition}

\begin{proof}
    From Theorem \ref{main}, we see that $\Pi$ is not decomposable. Using Remark \ref{rem1} and \ref{rem1:dec}, we conclude that $\Pi$ does not contain $G(2, W_{m-1})$ for any $W_{m-1} \in \V_{m-1}$. We now prove the assertion by induction on $m$. First, suppose $m = 4$ \footnote{this is a consequence of Corollary \ref{cor}, but we give a different proof} and let $W_3$ be a subspace of $V_4$ of dimension $3$. Since $\Pi$ does not contain $G(2, W_3)$ as mentioned above, and that $G(2, W_3)$ is a $2$-dimensional projective space, it follows that $|\Pi \cap G(2, W_3)| = q+1$. This is true for every $3$-dimensional subspace of $V_4$. Using Lemma \ref{zan}, we have
   $$|\Pi \cap G(2, V_4)| = (q+1) (q^2 + 1) = q^3 + q^2 + q + 1,$$
   which is equal to the desired upper bound\footnote{it follows that any nondecomposable hyperplane in $\PP\left( \bigwedge^2 V_4 \right)$ intersects the Grassmannian $G(2, V_4)$ at exactly $q^3 + q^2 + q + 1$ many points.}. Now assume that $m > 4$ and assume the result is true for $G(2, W_{m-1})$ for every subspace $W_{m-1} \in \V_{m-1}$.
   If $|G(2, W_{m-1}) \cap \Pi| = e(2, m-1)$, for some $W_{m-1} \in \V_{m-1}$, then we are done, thanks to Proposition \ref{submax}. Otherwise, induction hypothesis applies, and we have 
   $\displaystyle{|G(2, W_{m-1}) \cap \Pi| \le e' (2, m-1)}$ for every $W_{m-1} \in \V_{m-1}$.
   Using Lemma \ref{zan}, we deduce that
   $$|G(2, V_m) \cap \Pi| \le e' (2, m-1)\left(\frac{q^m - 1}{q^{m-2} - 1}\right) < e'(2, m).$$
   The last strict inequality follows from Proposition \ref{ineq}(c). It follows from the proof that the bound can possibly be attained by hyperplanes $\Pi$ such that $|G(2, W_{m-1}) \cap \Pi| = e(2, m-1)$, for some $W_{m-1}\in \V_{m-1}$. We refer to Remark \ref{rem:attained} for a class of hyperplanes attaining the bound. 
\end{proof}

We finally state the main result on the second minimum weight of the Grassmann codes. This is equivalent to the following Proposition. 
\begin{theorem}\label{sec}
    Let $\ell, m$ be positive integers with $2 \le \ell \le m-2$. Let $\Pi$ be a hyperplane in $\PP(\bigwedge^{\ell} V_m)$. If $|\Pi \cap G(\ell, V_m)| < e (\ell, m)$, then 
    $|G(\ell, V_m) \cap \Pi| \le e' (\ell, m).$
    Moreover, the bound is attained. 
\end{theorem}

\begin{proof}
    We prove this by induction on $m$. When $m = 4$, by hypothesis, $\ell = 2$. The assertion in this case follows from Proposition \ref{l2}.  Thus we may assume that $m > 4$ and the result is true for $G(\ell, W_{m-1})$ for every $W_{m-1} \in \V_{m-1}$. If $\ell = m-2$, then we are done, thanks to Corollary \ref{cor}. Thus, we may assume that $\ell \le m-3$. We distinguish the proof into two cases:
    
    \textbf{Case 1:} \emph{There exists $W_{m-1} \in \V_{m-1}$ such that $G(\ell, W_{m-1}) \subset \Pi$.} Using \eqref{section}, we may write  \begin{equation}\label{one}|G(\ell, V_m) \cap \Pi| = |G(\ell, W_{m-1})| + \sum_{\nu \in \Fq^{m-\ell}} |s^{-1} (\nu) \cap \Pi|.\end{equation} 
    If $\widecheck{F}$ is decomposable, then it follows from  Remark \ref{rem2} that $|s^{-1} (\nu) \cap \widecheck{\Pi}| = e(\ell - 1, m-1)$ for all $\nu \in \Fq^{m- \ell}$.  
    Proposition \ref{ineq} (b) would imply that $|G(\ell, m) \cap \Pi| = e (\ell, m)$, contradicting the hypothesis. Thus $\widecheck{F}$ is not decomposable. By induction hypothesis, we have
    $$|s^{-1} (\nu) \cap \Pi| = |G(\ell - 1, W_{m-1}) \cap \widecheck{\Pi}| \le e'(\ell - 1, m-1) \ \ \text{for all} \ \ \nu \in \Fq^{m- \ell}.$$ The assertion now follows from \eqref{one} and Proposition \ref{ineq} (d).
    
    \textbf{Case 2:} \emph{$\Pi$ does not contain $G(\ell, W_{m-1})$ for any $W_{m-1} \in \V_{m-1}$}. If there exists $W_{m-1} \in \V_{m-1}$ such that $|\Pi \cap G(\ell, W_{m-1})| = e(\ell, m-1)$, then we are done using Proposition \ref{submax}.  
    Otherwise, using induction hypothesis, for every $W_{m-1} \in \V_{m-1}$, we have $|G(\ell, W_{m-1}) \cap \Pi| \le e'(\ell, m-1)$. Now Lemma \ref{zan} applies, and the assertion is proved using Proposition \ref{ineq} (c). 
    Moreover, as we have seen in Remark \ref{rem:attained}, this bound is attained. Furthermore, this bound is also attained by another family of hyperplanes. For this, let $\Pi$ be a hyperplane containing $G(\ell, W_{m-1})$ such that $|\widecheck{\Pi} \cap G(\ell - 1, W_{m-1})| = e'(\ell - 1, {m-1})$. Proposition \ref{ineq} (d) and \eqref{one} imply that $\Pi$ intersects $G(\ell, V_m)$ at exactly $e' (\ell, m)$ many points. 
    
    Note that a hyperplane $\Pi$ containing $G(\ell, W_{m-1})$ satisfying $|\widecheck{\Pi} \cap G(\ell - 1, W_{m-1})| = e'(\ell - 1, m-1)$ attains the bound. A hyperplane satisfying the above conditions exists thanks to induction hypothesis on $m$. Moreover, a class of hyperplanes attaining the bound is provided in Remark \ref{rem:attained}. This completes the proof. 
    \end{proof}

\begin{remark}\normalfont
As explained in the Introduction, Theorem \ref{sec} is equivalent to the fact that the second minimum distance of $C(\ell, m)$ is given by $q^{\ell (m-\ell)} + q^{\ell (m - \ell) - 2}$. Note that in our proof, the minimum distance formula for the Schubert codes turned out to be instrumental. In particular, we have used the relevant result for the Schubert code $C_\theta (\ell, m)$. Since a classification of all the minimum weight codewords of the code $C_{\theta} (\ell, m)$ is not known, we were unable to classify all the codewords of $C(\ell, m)$ achieving the second minimum weight. In general, the determination of the weight distribution for $C(\ell, m)$ remains open. 
Thus it would be nice to prove the result without making use of the minimum distance formula for Schubert codes. On the other hand, it might be interesting to check whether the minimum weight formula for the Schubert codes could be obtained using our methodology, which may possibly give rise to a solution of the \emph{minimum distance codeword conjecture} framed by Ghorpade and Singh in \cite{GS}. We leave these questions for future research.  
\end{remark}
%\begin{enumerate}
   % \item[(a)] As an immediate corollary of Theorem \ref{main}, we conclude that the minimum distance of the Grassmann code $C(\ell, m)$ is given by $q^{\ell (m - \ell)}$. 
   % \item[(b)] Our proof shows that if $\Pi$ is a hyperplane of $\PP(\bigwedge^{\ell} V_m)$ that does not contain $G(\ell, V_{m-1})$ for any codimension $1$ subspace $V_{m-1}$ of $V_m$, then 
    %$$|\Pi \cap G(\ell, V_m)| \le   |G(\ell, V_m) \cap \Pi| \le  {m \brack \ell}_q - \left(\frac{q^m - 1}{q^{m-\ell} - 1}\right) q^{\ell(m-1 - \ell)}.$$
    %It is indeed interesting to know whether this is the best possible bound. We believe this will be an interesting research problem to be considered in the near future. 
    %\item[(c)] It turns out that the minimum distance $d(\ell, m)$ of a Grassmann code $C(\ell, m)$ satisfies the nice recursive identity $d(\ell, m) = q^{m-\ell} d(\ell - 1, m-1)$. Our decomposition of $G(\ell, V_m)$ and the proof give us a geometric intuition of the above recursive formula. We ask whether the second minimum weights of $C(\ell, m)$ satisfy similar recursions. Furthermore, would it be necessary for a hyperplane giving rise to a codeword that attains the second minimum weight to contain $G(\ell, V_{m-1})$ for some $(m-1)$-dimensional subspace $V_{m-1}$ of $V_m$?
   % \item[(d)] Will it be possible to employ a similar method to compute the minimum weight of Schubert codes and characterize the minimum weight codewords geometrically? 
  %\end{enumerate}  

\section{Acknowledgments}
The authors sincerely thank Trygve Johnsen for his careful reading of the manuscript and for providing some very interesting comments. The authors are also thankful to the anonymous referees for their careful reading of the manuscript and for their thoughtful comments, which greatly improved the presentation.

\end{document}